\documentclass[10pt]{article}
\usepackage{mathrsfs}
\usepackage{amsfonts}
\usepackage{amsthm,amsmath,amssymb,anysize}
\newtheorem{lemma}{Lemma}[section]
\newtheorem{theorem}[lemma]{Theorem}

\newtheorem{proposition}[lemma]{Proposition}
\newtheorem{definition}[lemma]{Definition}

\usepackage[numbers,sort&compress]{natbib}
\marginsize{3.6cm}{3.6cm}{3.6cm}{3cm}
\numberwithin{equation}{section}

\title{\textsf{ Local superderivations on Cartan type Lie superalgebras}}
\author{\textsc{Jixia Yuan,$^{1,2,}$}\footnote{Supported by  the National Natural Science Foundation of China (11601135), Natural Science Foundation of Heilongjiang Province of China (QC2017002) and Project funded by China Postdoctoral Science Foundation}\;\;\textsc{Liangyun Chen$^{1,}$}\footnote{Correspondence: chenly640@nenu.edu.cn. (L. Chen);  Supported by the National Natural Science Foundation
  of China (11771069, NSF of Jilin province (No. 20170101048JC) and the project of Jilin province department of education (No. JJKH20180005K))}\;\;  \textsc{ and Yan Cao$^{3,}$}\footnote{Supported by  the National Natural Science Foundation of China (11801121), Natural Science Foundation of Heilongjiang Province of China (QC2018006) } \\
  \\
  \textit{1. School of Mathematical and Statistics},
  \textit{Northeast Normal University} \\
  \textit{Changchun 130024, China}\\
\\
  \ \ \textit{2. School of Mathematical Sciences},
  \textit{Heilongjiang University} \\
  \textit{Harbin 150080, China}\\
  \\
  \ \ \textit{3.  Department of Mathematics},
  \textit{Harbin University of Science and Technology} \\
  \textit{Harbin 150080, China}
  }
\date{ }
\begin{document}
\maketitle
\begin{quotation}
\noindent\textbf{Abstract.} In this paper, we characterize the   local superderivations on  Cartan type Lie superalgebras over the complex field $\mathbb{C}$. Furthermore, we prove that every local   superderivations on Cartan type simple Lie superalgebras  is a  superderivations. As an application, using the results on local superderivations
we characterize the $2$-local superderivations on   Cartan type Lie superalgebras. We prove that every $2$-local   superderivations on Cartan type  Lie superalgebras  is a  superderivations.
\\



  \end{quotation}

  \setcounter{section}{0}
\section{Introduction}

As a natural generalization of Lie algebras, Lie superalgebras are closely related to many branches of mathematics.
 The classification  of all   finite dimensional simple Lie  superalgebras over an algebraically closed field of characteristic zero has been obtained by Kac \cite{k1}, which consists of classical   Lie superalgebras  and Cartan type Lie superalgebras.  Cartan type Lie superalgebras play an important role in the category of Lie superalgebras.   Cartan type Lie superalgebras over $\mathbb{C}$ are subalgebras of
the full superderivation algebras of the  exterior
superalgebras.  The structural theory of these superalgebras has been playing a key role in the theory of  Lie superalgebras.
Derivations and generalized derivations are very important notions in the
research of   algebras and their generalizations.

The concept of local derivation was introduced    in 1990 by Kadison \cite{k2},  Larson and   Sourour \cite{ls}, and the authors studied  local derivations of Banach algebra. In 2001 Johnson  showed that every local derivation from a $\mathbb{C}^{*}$-algebra $A$ into a Banach $A$-bimodule is a
derivation\cite{jl}. Local derivations on the
algebra $S(M,\tau)$ were studied deeply in paper \cite{aakn}. In recent years, local derivations have aroused the interest of a great many authors, see  \cite{c,hl,z}.  The  local derivations of Lie  algebras have been sufficiently studied. In 2016, the  local derivations of Lie algebras were proved  by Ayupov and Kudaybergenov\cite{ak}, and the authors proved that every local   derivation of a finite dimensional semisimple Lie algebras  over an algebraically closed field of characteristic zero is a  derivation. In 2018, Ayupov and Kudaybergenov showed that in the class of solvable Lie algebras there exist two facts. One is that local derivation is different from any other derivation and the second is that there indeed exists a kind of algebras in which each local derivations is a derivation\cite{ak1}.
In 2017,  Chen, Wang
and  Nan  mainly studied  local superderivations  on basic classical Lie superalgebras,  and the authors proved that every local superderivation  on basic classical Lie superalgebras except for $A(1, 1)$ over the complex number field $\mathbb{C}$ is a superderivation\cite{cwn}. In 2018,  Chen and  Wang
 studied  local superderivations  on Lie superalgebras $\mathfrak{q}(n)$,  and the authors proved  that every local superderivation on
 $\mathfrak{q}(n)$, $n>3$, is a superderivation\cite{cw}.

In this paper, we are interested in determining all local superderivations and $2$-local superderivations on Cartan type Lie superalgebras over $\mathbb{C}$.
Let $L$ be a Cartan type Lie superalgebra over $\mathbb{C}$.  The main result in this paper is a complete characterization of the local superderivations   on  $L$:
$$
\mathrm{LDer}(L)=\mathrm{Der}(L).
$$

The paper is organized as follows. In Section 2, we recall some necessary concepts and notations. In   Section 3, we  establish  several  lemmas, which will be used to characterize the local superderivations on Cartan type Lie superalgebras. In Section 4, we determine all local superderivations on Cartan type Lie superalgebras. In Section 5,  as an application, using the results on local superderivations  we determine all $2$-local superderivations on Cartan type Lie superalgebras.

\section{Preliminaries}
Throughout $\mathbb{C}$ is the  field of complex numbers, $\mathbb{N}$ the
set of nonnegative integers and
$\mathbb{Z}_2= \{\bar{0},\bar{1}\}$   the additive group of two
elements.  For a  vector superspace $V=V_{\bar{0}}\oplus
V_{\bar{1}}, $ we write $|x|$ for the
parity of    $x\in V_{\alpha},$ where $\alpha\in \mathbb{Z}_2$. Once the symbol $|x|$ appears
in this paper, it
 will imply that $x$ is  a  $\mathbb{Z}_2$-homogeneous element. We also adopt the following notation: For  a proposition $P$, put
 $\delta _{P}=1$ if $P$ is true and $\delta _{P}=0$ otherwise.

\subsection{Lie superalgebras, superderivation }
Let us  recall some definitions relative to Lie superalgebras and superderivations \cite{k1}.
\begin{definition}
A Lie superalgebra is a vector superspace $L=L_{\bar{0}}\oplus
L_{\bar{1}}$ with an even bilinear   mapping
$[\cdot,\cdot]: L\times L\longrightarrow L$   satisfying the following axioms:
\begin{eqnarray*}
&& [x,y]=-(-1)^{|x||y|}[y,x],\\
&&[x,[y,z]]=[[x,y],z]+(-1)^{|x||y|}[y,[x,z]]
\end{eqnarray*}
for all $x, y, z\in L.$
\end{definition}
\begin{definition}
We call a  linear map $D:  L \longrightarrow L$ a  superderivation of Lie superalgebra $L$ if it satisfies the
following   equation:
\begin{eqnarray*}\label{yuantange1}
&& D([x,y])=[D(x), y]+(-1)^{|D||x|}[x, D(y)]
\end{eqnarray*}
for all $x, y \in L.$
\end{definition}
Write $\mathrm{Der}_{\bar{0}}(L)$ (resp. $\mathrm{Der}_{\bar{1}}(L)$) for the set of all superderivations
of $\mathbb{Z}_{2}$-homogeneous $\bar{0}$ (resp.  $\bar{1}$) of $L$.  Denote
$$\mathrm{Der}(L)=\mathrm{Der}_{\bar{0}}(L)\oplus \mathrm{Der}_{\bar{1}}(L).$$

\subsection{local superderivation  and $2$-local superderivation}
Let us  recall some definitions relative to local superderivations  and 2-local superderivations \cite{cwn,wcn}. Let $L$ be a Lie superalgebra.
\begin{definition}
Recall that a linear map  $\phi: L\longrightarrow L$ is called a local superderivation if for every $x\in L$, there exists a superderivation $D_{x}\in \mathrm{Der}L$ (depending on $x$) such that $\phi(x)=D_{x}(x)$.
\end{definition}
\begin{definition}
Recall that a linear map  $\phi: L\longrightarrow L$ is called a $2$-local superderivation if for any two elements $x, y\in L$, there exists a superderivation $D_{x,y}\in \mathrm{Der}L$ (depending on $x, y$) such that $\phi(x)=D_{x,y}(x)$ and $\phi(y)=D_{x,y}(y)$.
\end{definition}
 A Local superderivation $\phi$ of $\mathbb{Z}_{2}$-homogeneous $\alpha$ of $L$ is a local superderivation such
that $\phi(L_{\beta})\subseteq L_{\alpha+\beta}$ for any $\beta\in \mathbb{Z}_{2}$. Write $\mathrm{LDer}_{\bar{0}}(L)$ (resp. $\mathrm{LDer}_{\bar{1}}(L)$) for the set of all super-biderivations
of $\mathbb{Z}_{2}$-homogeneous $\bar{0}$ (resp.  $\bar{1}$) of $L$.  Denote
$$\mathrm{LDer}(L)=\mathrm{LDer}_{\bar{0}}(L)\oplus \mathrm{LDer}_{\bar{1}}(L).$$

\subsection{Cartan type Lie superalgebras}

Let $n\geq 4$ be   an   integer and $\Lambda(n)$ be the exterior algebra in $n$ indeterminates $x_{1}, x_{2},\ldots, x_{n}$ with $\mathbb{Z}_{2}$-grading structure given by $|x_{i}|=\bar{1}.$   One may define a $\mathbb{Z}$-grading on $\Lambda(n)$ by letting $\mathrm{deg}\,x_{i}=1,$ where $1 \leq i\leq n.$ Write $n=2r$ or $n=2r+1,$ where $r\in \mathbb{N}.$ Put
 $\left[\frac{n}{2}\right]=r.$
 Cartan type Lie superalgebras consist of
four series of  simple Lie superalgebras contained in the full superderivation algebras of $\Lambda(n)$:
\begin{eqnarray*}
&&W(n)=\left\{\sum_{i=1}^{n}f_{i}\partial_{i}\mid  f_{i}\in \Lambda(n)\right\},
\\&&
S(n)=\left\{\sum_{i=1}^{n}f_{i}\partial_{i}\mid  f_{i}\in \Lambda(n), \sum_{i=1}^{n}\partial_{i}(f_{i})=0\right\},
\\&&
\widetilde{S}(n)=\left\{(1-x_{1}x_{2}\cdots x_{n})\sum_{i=1}^{n}f_{i}\partial_{i}\mid  f_{i}\in \Lambda(n), \sum_{i=1}^{n}\partial_{i}(f_{i})=0\right\}\; \mbox{(}n \;\mbox{is an even integer}\mbox{)},
\\&&
H(n)=\left\{\mathrm{D_{H}}(f)\mid  f\in \oplus_{i=0}^{n-1}\Lambda(n)_{i}\right\} \;\mbox{(}n> 4\mbox{)},
\end{eqnarray*}
where
$$\mathrm{D_{H}}(f)=(-1)^{|f|}\sum_{i=1}^{n}\partial_{i}(f)\partial_{i'},$$
\[ i'=\left\{
 \begin{array}{lll}
i+r,&\mbox{if  $1 \leq i\leq \left[\frac{n}{2}\right],$ }\\
i-r,&\mbox{if $ \left[\frac{n}{2}\right]< i\leq 2\left[\frac{n}{2}\right],$}\\
i,&\mbox{otherwise}.\\
\end {array}
\right.
\]

 One may define a $\mathbb{Z}$-grading on $W(n)$ by letting $\mathrm{deg}\,x_{i}=1=-\mathrm{deg}\,\partial_{i},$ where $1 \leq i\leq n.$ Thus $W(n)$   becomes a $\mathbb{Z}$-graded Lie superalgebra of 1 depth: $W(n)=\bigoplus_{j=-1}^{\xi_{W}}W(n)_{j},$ where $\xi_{W}=n-1.$
Suppose $L=S(n)$ or $H(n)$. Then $L$ is a  $\mathbb{Z}$-graded subalgebra  of $W(n)$. The $\mathbb{Z}$-grading is defined as follows:
  $L=\bigoplus_{j=-1}^{\xi_{L}}L_{j},$
  where $L_{j}=L\cap W(n)_{j}$ and
 \[
\xi_{L}=\left\{
 \begin{array}{ll}
n-2,&\mbox{if  $L=S(n),$ }\\
n-3,&\mbox{if $H(n)$.}
\end {array}
\right.
\]
Put $$\xi_{i}=x_{1}x_{2}\cdots x_{n}\partial_{i},\;\widetilde{S}(n)_{-1}=\mathrm{span}_{\mathbb{C}}\left\{\partial_{i}-\xi_{i}\mid  1\leq i \leq n\right\},\; \widetilde{S}(n)_{i}=S(n)_{i} \;\mbox{for\;} i>-1.$$
Then $\widetilde{S}(n)$ becomes a $\mathbb{Z}_{n}$-graded Lie superalgebra:
$
\widetilde{S}(n)=\bigoplus_{i=-1}^{\xi_{\widetilde{S}}}\widetilde{S}(n)_{i},
 $
 where $\xi_{\widetilde{S}}=n-2.$
The $0$-degree components of these superalgebras are classical Lie algebras:
$$W(n)_{0}\cong \mathfrak{gl}(n),\;S(n)_{0}=\widetilde{S}(n)_{0}\cong \mathfrak{sl}(n),\; H(n)_{0}\cong \mathfrak{so}(n).$$

Let $L=\oplus_{i\in \mathbb{Z}}L_{i}$ be a $\mathbb{Z}$-graded Lie superalgbra, $H_{L}$ be the standard Cartan subalgebra  of $L$, $\theta\in H_{L}^{*}$ be the zero root, $\Delta_{L}$ be the root system of $L$.
Let us describe the roots  of Cartan type Lie superalgebras. If  $L=W(n),$ we choose the standard basis $\{\varepsilon_{1}, \ldots, \varepsilon_{n}\}$ in $H_{L}^{*},$ and then
$$
\Delta_{L}=\{\varepsilon_{i_{1}}+\cdots +\varepsilon_{i_{k}}-\varepsilon_{i} \mid  1\leq i_{1} <\cdots< i_{k}\leq n,  1\leq i\leq n \}.
$$
The root systems of $S(n)$ and $\widetilde{S}(n)$ are obtained from the root system of $W(n)$
by removing the roots $\varepsilon_{1}+\cdots +\varepsilon_{n}-\varepsilon_{i},$ where $1\leq i \leq n.$
Finally if $L=H(n)$, then
$$\Delta_{L}=\left\{\pm \varepsilon_{i_{1}}\pm\cdots \pm\varepsilon_{i_{k}}  \mid  1\leq i_{1} <\cdots< i_{k}\leq \left[\frac{n}{2}\right]\right\}.$$

\section{General lemmas}
In this section, let us establish  several  lemmas,  which will be used to characterize the local superderivations on Cartan type Lie superalgebras. Put $\mathcal{C}=\sum_{i=1}^{n}x_{i}\partial_{i}$ and $\widetilde{H}(n)=\left\{\mathrm{D_{H}}(f)\mid  f\in \Lambda(n)\right\} $. By \cite{s}, we have the following lemma.
\begin{lemma}\label{l1}
Let $L$ be a Cartan type Lie superalgebra. Then  $\mathrm{Der}\,L=\mathrm{ad}L',$ where
\[L'\cong\left\{
 \begin{array}{lll}
L, &\mbox{if  $L=W(n), \widetilde{S}(n),$ }\\
L\oplus \mathbb{C}\mathcal{C}, &\mbox{if  $L=S(n),$ }\\
\widetilde{H}(n)\oplus \mathbb{C}\mathcal{C}, &\mbox{if  $L=H(n).$ }
\end {array}
\right. \]
\end{lemma}
Let $L$ be a  Cartan type Lie superalgebra. By Lemma \ref{l1} and a simple computation, we have
$\Delta_{L'}=\Delta_{L}$  and the following lemma.
\begin{lemma}\label{l2}
Let $L$ be a  Cartan type Lie superalgebra. Then
 $L'$ is transitive, that is $a\in \oplus_{i\geq 0}L'_{i} $ and  $\left[a, L'_{-1}\right]=0,$ then $a=0.$

\end{lemma}
Suppose $L$ is a Cartan type Lie superalgebra. For  $i\in \mathbb{Z}$ and $\alpha\in \Delta_{L'}$, we
put
$$
\mathrm{LDer}(L)_{i\times \alpha}=\{\phi\in \mathrm{LDer}(L)\mid  \;\mbox{there exists}\;  u_{x}^{i} \in L'_{i}\cap L'_{\alpha}  \;\mbox{such that}\;  \phi(x)=[u_{x}^{i}, x] \;\mbox{for all}\; x\in L\}.
$$
Then we have the following lemma.
\begin{lemma}\label{chenyuanl1}
Let $L$ be a Cartan type Lie superalgebra. Then  the following conclusion hold:
\begin{itemize}
  \item [(1)] If $L\neq \widetilde{S}(n)$, then $
\mathrm{LDer}(L)=\bigoplus\limits_{i\in \mathbb{Z}, \alpha\in \Delta_{L'}}\mathrm{LDer}(L)_{i\times\alpha}.
$
  \item [(2)] If $L= \widetilde{S}(n)$, then $
\mathrm{LDer}(L)=\bigoplus\limits_{i\in \mathbb{Z}_{n}, \alpha\in \Delta_{L'}}\mathrm{LDer}(L)_{i\times\alpha}.
$
\end{itemize}

 \end{lemma}
\begin{proof}
(1) By Lemma \ref{l1}, we have ``$\supseteq$'' part  is complete. Next, we verify  the ``$\subseteq$'' part. Let $\phi\in \mathrm{LDer}(L).$ For each $x\in L,$ by Lemma \ref{l1} there exists an element $u_{x}\in L'$ such that $\phi(x)=[u_{x}, x],$ where $u_{x}\in L'.$  Since $L'$ has the $\mathbb{Z}\times \Delta_{L'}$-grading,  we can write
$$u_{x}=\sum_{i\in \mathbb{Z}, \alpha\in \Delta_{L'}}u_{x}^{i,\alpha},$$
 where $u_{x}^{i,\alpha}\in L'_{i\times\alpha}.$ For $i\in \mathbb{Z}$ and  $\alpha\in \Delta_{L'},$ we set
 $$\phi_{i\times\alpha}(x)=[u_{x}^{i,\alpha}, x].$$
  A direct verification shows that $\phi_{i\times\alpha}\in \mathrm{LDer}(L)_{i\times\alpha}$ and
$$
\sum_{i\in \mathbb{Z},\alpha\in \Delta_{L'}}\phi_{i\times\alpha}(x)=\sum_{i\in \mathbb{Z},\alpha\in \Delta_{L'}}[u_{x}^{i,\alpha}, x]=[u_{x}, x]=\phi(x).
$$

(2)  A similar argument as for $L\neq \widetilde{S}(n)$ works also  for $L=\widetilde{S}(n).$

\end{proof}

\section{Local Superderivations  of Cartan type Lie superalgebras}
In this section we shall
characterize the local superderivations on Cartan type Lie superalgebras.
 Let $L$ be a Cartan type Lie superalgebra  and $\{h_{1}, \ldots, h_{l}\}$ be  the standard basis of $H_{L}.$ Set $h_{0}=\sum_{i=1}^{l}t^{i}h_{i},$ where $t$ is a fixed algebraic number from $\mathbb{C}$ of degree bigger than $l.$ Then we have the following propositions.
\begin{proposition}\label{chenl1}
Let $L$ be a Cartan type Lie superalgebra and $\phi\in \mathrm{LDer}(L).$ If $\phi(h_{0})=0,$ then $\phi\mid_{H_{L}}=0.$
\end{proposition}
\begin{proof}
For any $h_{i}(i=1, \ldots, l)$, there exists element
$$u^{i}=\sum_{1\leq k_{1}<\ldots< k_{l}\leq n}\sum_{j=1}^{n}a^{i}_{j,k_{1},\ldots, k_{l}}x_{k_{1}}\cdots x_{k_{l}}\partial_{j}\in L',$$
 where $a^{i}_{j,k_{1},\ldots, k_{l}}\in \mathbb{C}$ such that $\phi(h_{i})=[u^{i}, h_{i}].$
 Then
 $$
 \phi(h_{i})=-\sum_{1\leq k_{1}<\ldots< k_{l}\leq n}\sum_{j=1}^{n}a^{i}_{j,k_{1},\ldots, k_{l}}(\varepsilon_{k_{1}}+\cdots+\varepsilon_{k_{l}}-\varepsilon_{j})(h_{i})x_{k_{1}}\cdots x_{k_{l}}\partial_{j}.
 $$
 If $l\neq 1$ or $l=1$ and $k_{1}\neq j,$ then exists   $1\leq k \leq n$ such that $(\varepsilon_{k_{1}}+\cdots+\varepsilon_{k_{l}}-\varepsilon_{j})(h_{k})\neq 0.$ Put $$h_{ik}=(\varepsilon_{k_{1}}+\cdots+\varepsilon_{k_{l}}-\varepsilon_{j})(h_{k})h_{i}-(\varepsilon_{k_{1}}+\cdots+\varepsilon_{k_{l}}-\varepsilon_{j})(h_{i})h_{k}.$$
Then there exists element
$$u^{ik}=\sum_{1\leq k_{1}<\ldots< k_{l}\leq n}\sum_{j=1}^{n}a^{ik}_{j,k_{1},\ldots, k_{l}}x_{k_{1}}\cdots x_{k_{l}}\partial_{j}\in L',$$
 where $a^{ik}_{j,k_{1},\ldots, k_{l}}\in \mathbb{C}$ such that $\phi(h_{ik})=[u^{ik}, h_{ik}].$ Thus $\phi(h_{ik})\cap L_{\varepsilon_{k_{1}}+\cdots+\varepsilon_{k_{l}}-\varepsilon_{j}}=0.$ On the other hand,
 $$\phi(h_{ik})=(\varepsilon_{k_{1}}+\cdots+\varepsilon_{k_{l}}-\varepsilon_{j})(h_{k})\phi(h_{i})-(\varepsilon_{k_{1}}+\cdots+\varepsilon_{k_{l}}-\varepsilon_{j})(h_{i})\phi(h_{k}).$$
 Hence
\begin{eqnarray*}
&& (\varepsilon_{k_{1}}+\cdots+\varepsilon_{k_{l}}-\varepsilon_{j})(h_{k})(\varepsilon_{k_{1}}+\cdots+\varepsilon_{k_{l}}-\varepsilon_{j})(h_{i})a^{i}_{j,k_{1},\ldots, k_{l}}\\
&=&(\varepsilon_{k_{1}}+\cdots+\varepsilon_{k_{l}}-\varepsilon_{j})(h_{k})(\varepsilon_{k_{1}}+\cdots+\varepsilon_{k_{l}}-\varepsilon_{j})(h_{i})a^{k}_{j,k_{1},\ldots, k_{l}},
\end{eqnarray*}
that is
$$
(\varepsilon_{k_{1}}+\cdots+\varepsilon_{k_{l}}-\varepsilon_{j})(h_{i})a^{i}_{j,k_{1},\ldots, k_{l}}=(\varepsilon_{k_{1}}+\cdots+\varepsilon_{k_{l}}-\varepsilon_{j})(h_{i})a^{k}_{j,k_{1},\ldots, k_{l}}.
$$
Then
 $$
 \phi(h_{0})=\sum_{i=1}^{l}t^{i}\phi(h_{i})=-\sum_{i=1}^{l}\sum_{1\leq k_{1}<\ldots< k_{l}\leq n}\sum_{j=1}^{n}t^{i}a^{k}_{j,k_{1},\ldots, k_{l}}(\varepsilon_{k_{1}}+\cdots+\varepsilon_{k_{l}}-\varepsilon_{j})(h_{i})x_{k_{1}}\cdots x_{k_{l}}\partial_{j}.
 $$
 Since $\phi(h_{0})=0,$ we have $a^{k}_{j,k_{1},\ldots, k_{l}}\sum_{i=1}^{l}t^{i}(\varepsilon_{k_{1}}+\cdots+\varepsilon_{k_{l}}-\varepsilon_{j})(h_{i})=0.$
 Since  $t$ is a  algebraic number from $\mathbb{C}$ of degree bigger than $l$ and $(\varepsilon_{k_{1}}+\cdots+\varepsilon_{k_{l}}-\varepsilon_{j})(h_{i})$ are integers, then $\sum_{i=1}^{l}t^{i}(\varepsilon_{k_{1}}+\cdots+\varepsilon_{k_{l}}-\varepsilon_{j})(h_{i})\neq 0.$ Thus $a^{k}_{j,k_{1},\ldots, k_{l}}=0$ and $\phi(h_{i})=0.$ The proof is complete.
\end{proof}

To apply Lemma \ref{l2}, we give the following  proposition.
 \begin{proposition}\label{chenl2}
Let $L$ be a Cartan type Lie superalgebra,   $\phi\in \mathrm{LDer}(L)$  and  $\phi\mid_{L_{-1}\oplus H_{L}}=0$. Then  the following conclusion hold:
\begin{itemize}
  \item [(1)] If  $L\neq \widetilde{S}(n)$ and $\phi\in \mathrm{LDer}_{i\times \alpha}(L),$  where  $i\in \mathbb{Z}$  and $\alpha\in \Delta_{L'}$,  then $\phi$ is zero.
  \item [(2)] If  $L=\widetilde{S}(n)$ and $\phi\in \mathrm{LDer}_{i\times \alpha}(L),$  where  $i\in \mathbb{Z}_{n}$  and $\alpha\in \Delta_{L'}$,  then $\phi$ is zero.
\end{itemize}
\end{proposition}
\begin{proof}
For any $x\in \oplus_{j\geq 0}L_{j},$ there exists an element $u_{x}\in L'_{i}\cap L'_{\alpha}$ such that
\begin{eqnarray*}
\phi(x)&=&\phi(x+\partial_{1}+\cdots+ \partial_{n}-\delta_{L,\widetilde{S}}(\xi_{1}+\cdots+ \xi_{n}))\\
&=&[u_{x}, x+\partial_{1}+\cdots+ \partial_{n}-\delta_{L,\widetilde{S}}(\xi_{1}+\cdots+ \xi_{n})]\in \oplus_{j\geq 0}L_{i+j}.
\end{eqnarray*}
Since
$$[u_{x}, \partial_{1}+\cdots+ \partial_{n}-\delta_{L,\widetilde{S}}(\xi_{1}+\cdots+ \xi_{n})]\in L_{i-1}$$
 and $[u_{x}, x]\in \oplus_{j\geq 0}L_{i+j}$, then
 $$[u_{x}, \partial_{1}+\cdots+ \partial_{n}-\delta_{L,\widetilde{S}}(\xi_{1}+\cdots+ \xi_{n})]=0.$$
  Because $\partial_{1}-\delta_{L,\widetilde{S}}\xi_{1}, \ldots, \partial_{n}-\delta_{L,\widetilde{S}} \xi_{n}$ belong to different root space, so  $[u_{x}, \partial_{j}-\delta_{L,\widetilde{S}}\xi_{j}]=0$ for all $1\leq j \leq n.$
By Lemma \ref{l2},  we have $u_{x}\in L'_{-1}.$ Note that   every root space of $L'_{-1}$ is one dimension. Then there is $1\leq k \leq l$ and $a_{x}\in \mathbb{C}$ such that $u_{x}=a_{x}(\partial_{k}-\delta_{L,\widetilde{S}}\xi_{k}).$
Take $h\in H_{L}$ satisfied $[(\partial_{k}-\delta_{L,\widetilde{S}}\xi_{k}), h]\neq 0.$ By
\begin{eqnarray*}\label{e2}
[a_{x}(\partial_{k}-\delta_{L,\widetilde{S}}\xi_{k}), x]=\phi(x)=\phi(h+x)=[a_{h+x}(\partial_{k}-\delta_{L,\widetilde{S}}\xi_{k}), h+x],
\end{eqnarray*}
we have $a_{h+x}[(\partial_{k}-\delta_{L,\widetilde{S}}\xi_{k}), h]=0$ for all $x\in \oplus_{i\geq 1}L_{i}$ or $x\in L_{0}\cap L_{\beta},$ where $\theta \neq \beta\in \Delta_{L}.$ Then $a_{h+x}=0,$ that is $\phi(x)=0$ for all    $x\in \oplus_{i\geq 1}L_{i}$ or $x\in L_{0}\cap L_{\beta},$ where $\theta \neq \beta\in \Delta_{L}.$ This together with $\phi\mid_{L_{-1}\oplus H_{L}}=0$ implies that $\phi=0.$ The proof is complete.
\end{proof}
By Proposition \ref{chenl2},  we have the following   local superderivations vanishing  propositions.
\begin{proposition}\label{pro4.1}
Let $L$ be a Cartan type Lie superalgebra,   $\phi\in \mathrm{LDer}(L)$  and  $\phi\mid_{H_{L}}=0$. Then  the following conclusion hold:
\begin{itemize}
  \item [(1)] If  $L\neq \widetilde{S}(n)$ and $\phi\in \mathrm{LDer}_{i\times \alpha}(L),$  where  $i\in \mathbb{Z}$  and $\theta\neq\alpha\in \Delta_{L'}$,  then $\phi$ is zero.
  \item [(2)] If  $L=\widetilde{S}(n)$ and $\phi\in \mathrm{LDer}_{i\times \alpha}(L),$  where  $i\in \mathbb{Z}_{n}$  and $\theta\neq\alpha\in \Delta_{L'}$,  then $\phi$ is zero.
\end{itemize}
\end{proposition}
\begin{proof}
Since $\alpha\neq \theta,$ there exists an element $h\in H_{L}$ such that $\alpha(h)\neq 0.$  Let   $1\leq k\leq n$. By Lemma  \ref{l1}, we know that there is $u_{k}, v_{k}\in L'_{\alpha}$ such that
$$
\phi(\partial_{k}-\delta_{L,\widetilde{S}}\xi_{k})=[u_{k}, \partial_{k}-\delta_{L,\widetilde{S}}\xi_{k}],
$$
$$\phi(h+\partial_{k}-\delta_{L,\widetilde{S}}\xi_{k})=[v_{k}, h+\partial_{k}-\delta_{L,\widetilde{S}}\xi_{k}].$$
 Since $\phi(H_{L})=0,$
\begin{eqnarray*}
&&[u_{k}, \partial_{k}-\delta_{L,\widetilde{S}}\xi_{k}]=\phi(\partial_{k}-\delta_{L,\widetilde{S}}\xi_{k})=\phi(h+\partial_{k}-\delta_{L,\widetilde{S}}\xi_{k})\\
&=&[v_{k}, h+\partial_{k}-\delta_{L,\widetilde{S}}\xi_{k}]=-\alpha(h)v_{k}+[v_{k}, \partial_{k}-\delta_{L,\widetilde{S}}\xi_{k}].
\end{eqnarray*}
Since $\alpha\neq \theta$,  we have $v_{k}=0.$ Then
$$\phi(\partial_{k}-\delta_{L,\widetilde{S}}\xi_{k})=\phi(h+\partial_{k}-\delta_{L,\widetilde{S}}\xi_{k})=[v_{k}, h+\partial_{k}-\delta_{L,\widetilde{S}}\xi_{k}]=0,$$
that is $\phi\mid_{L_{-1}}=0.$ By Proposition  \ref{chenl2}, we have $\phi=0.$
\end{proof}

\begin{proposition}\label{pro4.2}
Let $L$ be a Cartan type Lie superalgebra and  $\phi\in \mathrm{LDer}(L)$. Then  the following conclusion hold:
\begin{itemize}
  \item [(1)] If  $L\neq \widetilde{S}(n)$ and $\phi\in \mathrm{LDer}_{i\times \theta}(L),$  where  $i\in \mathbb{Z}$,  then $\phi$ is a superderivation.
  \item [(2)] If  $L=\widetilde{S}(n)$ and $\phi\in \mathrm{LDer}_{i\times \theta}(L),$  where  $i\in \mathbb{Z}_{n}$,  then $\phi$ is a superderivation.
\end{itemize}
\end{proposition}
\begin{proof}
Let $1\leq k\leq n.$ For $\partial_{k}-\delta_{L,\widetilde{S}}\xi_{k}$, there exists an element $u_{k}\in L'_{\theta}$ such that $$\phi(\partial_{k}-\delta_{L,\widetilde{S}}\xi_{k})=[u_{k}, \partial_{k}-\delta_{L,\widetilde{S}}\xi_{k}].$$
 For $\partial_{1}+\cdots+ \partial_{n}-\delta_{L,\widetilde{S}}(\xi_{1}+\cdots+ \xi_{n})$, there exists an element $u\in L'_{\theta}$ such that
 $$
 \phi(\partial_{1}+\cdots+ \partial_{n}-\delta_{L,\widetilde{S}}(\xi_{1}+\cdots+ \xi_{n}))=[u, \partial_{1}+\cdots+ \partial_{n}-\delta_{L,\widetilde{S}}(\xi_{1}+\cdots+ \xi_{n})].
 $$
Since $\phi$ is a linear map,
$$
[u, \partial_{1}+\cdots+ \partial_{n}-\delta_{L,\widetilde{S}}(\xi_{1}+\cdots+ \xi_{n})]=[u_{1}, \partial_{1}-\delta_{L,\widetilde{S}}\xi_{1}]+\cdots+[u_{n}, \partial_{n}-\delta_{L,\widetilde{S}}\xi_{n}].
$$
Then
$$
[u, \partial_{k}-\delta_{L,\widetilde{S}}\xi_{k}]=[u_{i}, \partial_{k}-\delta_{L,\widetilde{S}}\xi_{k}]
$$
for all $1\leq k\leq n.$ Put $\phi'=\phi-\mathrm{ad}u.$ Then $\phi'(\partial_{k}-\delta_{L,\widetilde{S}}\xi_{k})=0$ for all $1\leq k\leq n,$ that is $\phi'\mid_{L_{-1}}=0.$ Since $\phi\in \mathrm{LDer}_{i\times \theta}(L)$, we have $\phi'\mid_{H_{L}}=0$. By Proposition  \ref{chenl2}, we have $\phi'=0.$ Then   $\phi=\mathrm{ad}u$ is a superderivation.
 The proof is complete.
\end{proof}
By Propositions \ref{chenl1}, \ref{pro4.1} and \ref{pro4.2}, we have the following theorem.
\begin{theorem}\label{t1}
Let $L$ be a Cartan type Lie superalgebra. Then
$$\mathrm{LDer}(L)=\mathrm{Der}(L).$$
\end{theorem}
\begin{proof}
Only the ``$\subseteq$'' part needs a verification. Let $\phi\in \mathrm{LDer}(L).$ By Lemma \ref{l1},  there exists an element
$u\in L'$ such that $\phi(h_{0})=[u, h_{0}].$ Put $\phi'=\phi-\mathrm{ad}u.$ Then $\phi'(h_{0})=0.$
By Proposition \ref{chenl1}, we have $\phi'\mid_{H_{L}}=0.$ Propositions \ref{pro4.1} and \ref{pro4.2}  together with Lemma
 \ref{chenyuanl1} implies that $\phi'\in \mathrm{Der}(L).$ Therefore, $\phi\in \mathrm{Der}(L).$
\end{proof}

\section{Applications}
In this  section, we will characterize the $2$-local superderivation of    Cartan type Lie superalgebras. In \cite{s1} P. $\breve{S}$emrl introduced the concept of $2$-local derivations. Moreover, the author proved that every $2$-local derivation on B(H) is a derivation. Similarly, some authors started to describe  $2$-local derivation.  In \cite{kk} S. Kim and J. Kim give a short proof of that every 2-local derivation on the algebra $M_{n}(\mathbb{C})$ is a derivation.  A similar description for the finite-dimensional case appeared later in  \cite{m}. In the paper \cite{lw} 2-local derivations and automorphisms have been described on matrix algebras over finite-dimensional division rings.   Later J. Zhang and H. Li \cite{zl} extended the above result for arbitrary symmetric digraph matrix algebras and construct an example of 2-local derivation which is not a derivation on the algebra of all upper triangular complex $2\times2$ matrices. £ In \cite{ff}  Fo$\breve{s}$ner introduced the concept of $2$-local superderivations on the associctive superalgebra and  the authors proved that every $2$-local superderivation on superalgebra $M_{n}(\mathbb{C})$ is a superderivation.
In 2017,  Chen, Wang
and  Nan  mainly studied   $2$-local superderivations on basic classical Lie Superalgebras,   the authors proved that every  $2$-local superderivations on basic classical Lie superalgebras except for $A(1, 1)$ over the complex number field $\mathbb{C}$ is a superderivation\cite{wcn}.

Using the results on local superderivations  we have the following theorem.
\begin{theorem}
Let $L$ be a Cartan type Lie superalgebra. Then
every $2$-local superderivation of   $L$ is a superderivation.
\end{theorem}
\begin{proof}
 By the definition of $2$-local superderivation, we know that   every $2$-local superderivation of $L$ is  a local superderivation. Let $\phi$ is an $2$-local superderivation of   $L$. Then $\phi\in \mathrm{LDer}(L)$. By Theorem \ref{t1}, we have $\phi\in \mathrm{Der}(L).$
\end{proof}


\end{document}